\documentclass[12pt, oneside, reqno]{article}
\usepackage{amsmath,amsthm}
\usepackage{amssymb,latexsym}
\usepackage{enumerate}
\usepackage{cases}

\usepackage{tikz}
\newtheorem{thm}{Theorem}[section]

\newtheorem{lem}[thm]{Lemma}
\newtheorem{prop}[thm]{Proposition}

\theoremstyle{definition}

\newtheorem{rem}[thm]{Remark}

\numberwithin{equation}{section}

\begin{document}

\title{\Large Some inverse results of sumsets
}\author{\large Min Tang\thanks{Corresponding author. This work was supported by the National Natural Science Foundation of China(Grant
No. 11971033).} and Yun Xing}
\date{} \maketitle
 \vskip -3cm
\begin{center}
\vskip -1cm { \small
\begin{center}
 School of Mathematics and Statistics, Anhui Normal
University
\end{center}
\begin{center}
Wuhu 241002, PR China
\end{center}}
\end{center}

  {\bf Abstract:} Let $h\geq 2$ and $A=\{a_0,a_1,\ldots,a_{k-1}\}$ be a finite set of integers. It is well-known that $\left|hA\right|=hk-h+1$ if and only if $A$ is a $k$-term arithmetic progression.
  In this paper, we give some nontrivial inverse results of the sets $A$ with some extremal
 the cardinalities of $hA$.

{\bf Keywords:} sumsets; inverse problem

2010 Mathematics Subject Classification: 11B13\vskip8mm

\section{Introduction}

Let $[a,b]$ denote the interval of integers $n$ such that $a\leq n\leq b$. Let $A=\{a_0,a_1,\ldots,a_{k-1}\}$ be a finite set of integers such that
$a_0<a_1<\ldots<a_{k-1}$, we define
$$d(A)=\gcd(a_1-a_0, a_2-a_0,\ldots,a_{k-1}-a_0).$$
Let $a_i'=(a_i-a_0)/d(A),\; i=0,1,\ldots,k-1$. We call
$$A^{(N)}=\{a_0',a_1',\ldots,a_{k-1}'\}$$
the normal form of the set $A$. For any integer $c$, we define the set
$$c+A=\{c+A: a\in A\}.$$
For any finite set of integers $A$ and any positive integer $h\geq 2$, let $hA$ be the set of all sums of $h$ elements of $A$, with repetitions allowed.
It is easy to see that $\left|hA\right|=\left|hA^{(N)}\right|$.
For given set $A$, a direct problem is to determine the structure and properties of the $h$-fold sumset $hA$ when the set $A$ is known. An inverse problem is to deduce properties of the set $A$ from properties of the sumset $hA$.

The following two results gave the simple lower bound of the cardinality of $hA$ and showed that the lower bound is attained if and only if the set is an arithmetic progression.

\noindent{\bf Theorem A}(\cite{Nathanson96}, Theorem 1.3) Let $h\geq 2$. Let $A$ be a finite set of integers with $\left|A\right|=k$. Then
$$\left|hA\right|\geq hk-h+1.$$

\noindent{\bf Theorem B}(\cite{Nathanson96}, Theorem 1.6). Let $h\geq 2$. Let $A$ be a finite set of integers with $\left|A\right|=k$. Then $\left|hA\right|=hk-h+1$ if and only if $A$ is a $k$-term arithmetic progression.

In 1959, Freiman \cite{Freiman} proved the following result:

\noindent{\bf Theorem C} Let $k\geq 3$. Let $A=\{a_{0}, a_{1}, \ldots, a_{k-1}\}$ be a set of integers such that
$0 = a_{0} < a_{1} < \cdots < a_{k-1}$ and $\gcd(A)=1$. Then
\begin{displaymath}
\displaystyle\left|2A\right|\geq \min \{a_{k-1}, 2k-3\}+k=
\begin{cases}
a_{k-1}+k, \quad&\; \text{ if } a_{k-1}\leq 2k-3,\\
3k-3,\quad&\; \text{ if } a_{k-1}\geq 2k-2.
\end{cases}
\end{displaymath}
In \cite{chen1997}, \cite{Freiman62}, \cite{Lev1995}, \cite{Stanchescu}, the authors generalized the above theorem to the case of summation of two distinct sets.
In 1959, Freiman \cite{Freiman}(see also \cite{Nathanson96}) investigated the structure of set $A$ if the cardinality of $2A$
is between $2k-1$ and $3k-4$.

\noindent{\bf Theorem D}(\cite{Nathanson96}, Theorem 1.16) Let $A$ be a set of integers such that $\left|A\right|=k\geq 3$. If
$\left|2A\right|=2k-1+b\leq 3k-4,$
then $A$ is a subset of an arithmetic progression of length $k+b\leq 2k-3$.

In 1996, Lev \cite{Lev1996} gave the following result:

\noindent{\bf Theorem E}(\cite{Lev1996}, Theorem 1) Let $h$, $k\geq 2$ be integers. Let $A=\{a_{0}, a_{1}, \ldots, a_{k-1}\}$ be a set of integers such that
$0 = a_{0} < a_{1} < \cdots < a_{k-1}$ and $\gcd(A)=1$. Then
$$\left|hA\right|\geq\left|(h-1)A\right|+\min\{a_{k-1},h(k-2)+1\}.$$
For other related problems, see \cite{Freiman0901}-\cite{Lev2000}, \cite{Nathanson95}-\cite{Nathanson}, \cite{Yang}.

In this paper, we consider the following inverse problem: assume that $A$ is a finite integer set and the cardinalities of $hA$ are extremal cases, how to determine the structure of the set $A$?
We obtain the following results:

\begin{thm}\label{thm1} Let $h\geq 2$ and $k\geq 5$ be integers. Let $A$ be an integer set with $\left|A\right|=k$. If $hk-h+1<|hA|\leq hk+h-2$, then
$$A^{(N)}=[0,k]\setminus \{i\}, \quad 1\leq i\leq k-1.$$
Moreover, $|hA|=hk$ for $i=1$ or $k-1$, and $|hA|=hk+1$ for $2\leq i\leq k-2$.

\end{thm}

\begin{thm}\label{thm2} Let $h\geq 2$ and $k\geq 5$ be integers. Let $A$ be an integer set with $\left|A\right|=k$. If $hk+h-2<|hA|\leq hk+2h-3$, then
$$A^{(N)}=[0,k+1]\setminus \{i,j\}\quad 1\leq i<j\leq k+1.$$
Moreover, we have

\noindent (a) $|hA|=hk+h-1$ for $\{i,j\}=\{1,2\}, \{k-1,k\},\{1,k\},\{1,3\},\{k-2,k\};$

\noindent (b) $|hA|=hk+h$ for $i=1$ and $4\leq j\leq k-1$ when $h\geq 2$; or $2\leq i\leq k-3$ and $j=k$ when $h\geq 2$, or $\{i,j\}=\{2,3\}, \{k-2,k-1\}$ when $h=2$;

\noindent (c) $|hA|=hk+h+1$ for $2\leq i<j\leq k-1$, except for $\{i,j\}=\{2,3\}, \{k-2,k-1\}$ when $h=2$.
\end{thm}

\begin{rem} By Theorem \ref{thm1} and Theorem \ref{thm2} we know that there is no set $A$ such that $|3A|=3k-1$.
\end{rem}
\section{Lemmas}

\begin{lem}\label{lem1}Let $h\geq 2$ and $k\geq 5$ be integers. Let $A=\{a_{0}, a_{1}, \ldots, a_{k-1}\}$ be a set of integers such that
$0 = a_{0} < a_{1} < \cdots < a_{k-1}$ and $\gcd(A)=1$.
 If $|hA|\leq hk+2h-3$, then $a_{k-1}\leq k+1$. Moreover,
 if $|hA|\leq hk+h-2$, then $a_{k-1}\leq k$.
\end{lem}
\begin{proof} By Theorem E, we have \begin{eqnarray}\label{2.1}\left|hA\right|&\geq&\left|(h-1)A\right|+\min\{a_{k-1},h(k-2)+1\}\\
&\geq&\left|(h-2)A\right|+\min\{a_{k-1},h(k-2)+1\}+\min\{a_{k-1},(h-1)(k-2)+1\}\nonumber\\
&\geq& \cdots\cdots\nonumber\\
&\geq&\left|A\right|+\min\{a_{k-1},h(k-2)+1\}+\cdots+\min\{a_{k-1},2(k-2)+1\}.
\nonumber
\end{eqnarray}

 If $|hA|\leq hk+2h-3$, then $a_{k-1}\leq 2(k-2)+1$. Otherwise, if $a_{k-1}\geq 2k-2$, then by $(\ref{2.1})$ and $k\geq 5$, we have $$|hA|\geq k+(h-2)(2k-2)+2k-3>hk+2h-3,$$
which is impossible. Thus, again by (\ref{2.1}) we have
$$hk+2h-3\geq |hA|\geq k+(h-1)a_{k-1},$$
hence $a_{k-1}\leq k+1$.

If $|hA|\leq hk+h-2$, then by the above discussion, we have $a_{k-1}\leq 2(k-2)+1$. Thus, by (\ref{2.1}) we have
$hk+h-2\geq |hA|\geq k+(h-1)a_{k-1},$
hence $a_{k-1}\leq k$.

This completes the proof of Lemma \ref{lem1}.
\end{proof}

\begin{lem}\label{lem2} Let $i,j$ be positive integers such that $i\geq 2$ and $j\geq i+2$. Put $A=[0,i-1]\cup [i+1,j]$. Then
$hA=[0,hj]$ for all $h\geq 2$.
\end{lem}
\begin{proof}We have
\begin{equation}{\label{2.3}}[0,hi-h]\cup [hi+h,hj]\subset hA.\end{equation}
Write $$A_1=\{i-2,i-1\}, \quad A_2=\{i+1,i+2\}.$$ Since $i\geq 2$ and $j\geq i+2$, we have $A_1\cup A_2\subset A$.

For $h\geq 2$, we have $hi-h+3l+1\geq hi-2h+3(l+1)$ for all $0\leq l\leq h$. Thus
\begin{eqnarray}{\label{2.4}}
h(A_1\cup A_2)&=&\bigcup\limits_{l=0}^h\left((h-l)A_1+l A_2\right)\nonumber\\
&=&\bigcup\limits_{l=0}^h\big([(i-2)(h-l),(i-1)(h-l)]+[l(i+1),l(i+2)]\big)\nonumber\\
&=&\bigcup\limits_{l=0}^h[hi-2h+3l,hi-h+3l]\\
&=& [hi-2h,hi+2h].\nonumber
\end{eqnarray}
By (\ref{2.3}) and (\ref{2.4}), we have $hA=[0,hj]$.

This completes the proof of Lemma \ref{lem2}.
\end{proof}

\begin{lem}\label{lem3} Let $i,j$ be positive integers such that $i\geq 2$ and $j\geq i+3$. Put $A=[0,i-1]\cup [i+2,j]$. Then
$hA=[0,hj]$ for all $h\geq 3$.
\end{lem}
\begin{proof}We have
\begin{equation}{\label{2.5}}[0,hi-h]\cup [hi+2h,hj]\subset hA.\end{equation}
Write $$A_1=\{i-2,i-1\}, \quad A_2=\{i+2,i+3\}.$$ Since $i\geq 2$ and $j\geq i+3$, we have $A_1\cup A_2\subset A$.

For $h\geq 3$, we have $hi-h+4l+1\geq hi-2h+4(l+1)$ for all $0\leq l\leq h$. Thus
\begin{eqnarray}{\label{2.6}}
h(A_1\cup A_2)&=&\bigcup\limits_{l=0}^h\left((h-l)A_1+l A_2\right)\nonumber\\
&=&\bigcup\limits_{l=0}^h\big([(i-2)(h-l),(i-1)(h-l)]+[l(i+2),l(i+3)]\big)\nonumber\\
&=&\bigcup\limits_{l=0}^h[hi-2h+4l,hi-h+4l]\\
&=& [hi-2h,hi+3h].\nonumber
\end{eqnarray}
By (\ref{2.5}) and (\ref{2.6}), we have $hA=[0,hj]$.

This completes the proof of Lemma \ref{lem3}.
\end{proof}

\section{Propositions}
\begin{prop}\label{prop1} Let $h\geq 2$, $k\geq 4$ be positive integers and $A^{(N)}=[0,k]\backslash\{i\}$ for some $i\in[1,k]$. Then

\noindent(1) If $i=k$, then $\left|hA^{(N)}\right|=hk-h+1$;

\noindent(2) If $i=1$ or $k-1$, then $\left|hA^{(N)}\right|=hk$;

\noindent(3) If $2\leq i\leq k-2$, then $\left|hA^{(N)}\right|=hk+1$.
\end{prop}

\begin{proof} (1) If $A^{(N)}=[0,k-1]$, then by Theorem B, we have $\left|hA^{(N)}\right|=hk-h+1$.

(2) If $i=1$, then $A^{(N)}=\{0\}\cup [2,k].$ We have $$1\notin hA^{(N)}, \;\{0\}\cup [2h,hk]\subset hA^{(N)}.$$
For $2\leq m\leq 2h-1$, let $r_m$ be the least nonnegative residue of $m$ modulo $2$, we have $2+r_m\in A^{(N)}$ and
\begin{equation*}
  \begin{array}{lrrrl}
  m=&\underbrace{2+\cdots +2}+&\underbrace{0+\cdots+0}+(2+r_m). \\
  & \multicolumn{1}{c}{\frac{m-r_m}{2}-1 \text{ copies}}&  \multicolumn{2}{c}{ h-\frac{m-r_m}{2} \text{ copies}}
  \end{array}
\end{equation*}
Hence, we have $|hA^{(N)}|=hk$.

If $i=k-1$, then by $$A^{(N)}=[0, k-2]\cup \{k\}=k-(\{0\}\cup [2,k]),$$
we have $\left|hA^{(N)}\right|=hk$.

(3) If $2\leq i\leq k-2$, then
$A^{(N)}=[0,i-1]\cup[i+1,k].$
By Lemma \ref{lem2}, we have $hA^{(N)}=[0,hk]$. Thus
$\left|hA^{(N)}\right|=hk+1.$

This completes the proof of Proposition \ref{prop1}.
\end{proof}

\begin{prop}\label{prop2} Let $h\geq 2$, $k\geq 5$ be positive integers and $A^{(N)}=[0,k+1]\backslash\{i,i+1\}$ for some $i\in[1,k]$. Then

\noindent(1) If $i=k$, then $\left|hA^{(N)}\right|=hk-h+1$;

\noindent(2) If $i=1$ or $k-1$, then $\left|hA^{(N)}\right|=hk+h-1$;

\noindent(3) If $2\leq i\leq k-2$, then $\left|hA^{(N)}\right|=hk+h+1$ for $h\geq 3$. For $h=2$ and $i=2$ or $k-2$, we have $\left|2A^{(N)}\right|=2k+2$;
For $h=2$ and $3\leq i\leq k-3$, we have $\left|2A^{(N)}\right|=2k+3$.
\end{prop}

\begin{proof} (1) If $A^{(N)}=[0,k-1]$, then by Theorem B, we have $\left|hA^{(N)}\right|=hk-h+1$.

(2) If $A^{(N)}=\{0\}\cup[3,k+1]$, then $$1,2\notin hA^{(N)}, \;\{0\}\cup [3h,hk+h]\subset hA^{(N)}.$$
For $3\leq m\leq3h-1$, let $r_m$ be the least nonnegative residue of $m$ modulo $3$. Noting that $r_m+3\in A^{(N)}$ and $\displaystyle 1\leq \left\lfloor \frac{m-r_m}{3}\right\rfloor\leq h-1$, we have
\begin{equation*}
  \begin{array}{lrrrl}
 m=&\underbrace{3+\cdots +3}+&(3+r_m)+&\underbrace{0+\cdots+0}.\\
 & \multicolumn{1}{c}{\frac{m-r_m}{3}-1 \text{ copies}}&  &\multicolumn{2}{c}{h-\frac{m-r_m}{3} \text{ copies}}
  \end{array}
\end{equation*}
 Hence, $|hA^{(N)}|=hk+h-1$.

 If $A^{(N)}=[0,k-2]\cup \{k+1\}$, then by $$A^{(N)}=(k+1)-\left(\{0\}\cup[3,k+1]\right),$$ we have $|hA^{(N)}|=hk+h-1$.

(3) If $2\leq i\leq k-2$, then $$A^{(N)}=[0,i-1]\cup [i+2,k+1].$$

If $h\geq 3$, then by Lemma \ref{lem3} we have $A^{(N)}=[0,hk+h]$, thus $|hA^{(N)}|=hk+h+1$.

If $i=2$ and $h=2$, then $|2A^{(N)}|= 2k+2$. If $3\leq i\leq k-2$ and $h=2$, then
$$2A^{(N)}=[0,2i-2]\cup [i+2,k+i]\cup [2i+4,2k+2].$$
If $i\leq k-3$, then $2A^{(N)}=[0,2k+2]$; if $i=k-2$, then $2A^{(N)}=[0,2k-2]\cup [2k,2k+2]$.
Hence, $|2A^{(N)}|=2k+2$ or $2k+3$.

This completes the proof of Proposition \ref{prop2}.
\end{proof}

\begin{prop}\label{prop3} Let $h\geq 2$, $k\geq 5$ be positive integers and $A^{(N)}=[0,k+1]\backslash\{i,i+2\}$ for some $i\in[1,k-1]$. Then

\noindent(1) If $i=k-1$, then $\left|hA^{(N)}\right|=hk$;

\noindent(2) If $i=1$ or $k-2$, then $\left|hA^{(N)}\right|=hk+h-1$;

\noindent(3) If $2\leq i\leq k-3$, then $\left|hA^{(N)}\right|=hk+h+1$.

\end{prop}

\begin{proof} (1) If $i=k-1$, then $A^{(N)}=[0,k-2]\cup\{k\}$. By Proposition \ref{prop1}(2), we have $|hA^{(N)}|=hk$.

(2) If $i=1$, then $A^{(N)}=\{0\}\cup\{2\}\cup[4,k+1]$. We have $$1,3\notin hA^{(N)}, \{0,2\}\cup [4h,hk+h]\subset hA^{(N)}.$$

For $4\leq m\leq 4h-1$, let $r_m$ be the least nonnegative residue of $m$ modulo $4$. Then $\displaystyle 1\leq \left\lfloor \frac{m-r_m}{4}\right\rfloor\leq h-1$.
If $r_m=0$ or $1$, then
\begin{equation*}
  \begin{array}{lrrrl}
  m=&\underbrace{4+\cdots +4}+&\underbrace{0+\cdots+0}+(4+r_m). \\
  & \multicolumn{1}{c}{\frac{m-r_m}{4}-1 \text{ copies}}&  \multicolumn{2}{c}{h-\frac{m-r_m}{4} \text{ copies}}
  \end{array}
\end{equation*}
If $r_m=2$ or $3$, then
\begin{equation*}
  \begin{array}{lrrrl}
  m=&\underbrace{4+\cdots +4}+&\underbrace{0+\cdots+0}+2+(2+r_m). \\
  & \multicolumn{1}{c}{\frac{m-r_m}{4}-1 \text{ copies}}&  \multicolumn{2}{c}{h-\frac{m-r_m}{4}-1 \text{ copies}}
  \end{array}
\end{equation*}
Hence, $|hA^{(N)}|=hk+h-1$.

If $i=k-2$, then $$A^{(N)}=(k+1)-\left(\{0\}\cup\{2\}\cup[4,k+1]\right).$$ Thus $|hA^{(N)}|=hk+h-1$.

(3) If $2\leq i\leq k-3$, then $A^{(N)}=[0,i-1]\cup\{i+1\}\cup[i+3,k+1].$
Thus $$[0,hi-h]\cup \{hi+h\}\cup \{hi+3h,hk+h\}\subset hA^{(N)}.$$

Now we shall show that $h(i-1)+m, h(i+1)+m\in hA^{(N)}$ for $1\leq m\leq 2h-1$.
For $m=1$ we have
\begin{equation*}
  \begin{array}{rrrr}
  &h(i-1)+1=\underbrace{(i-1)+\cdots +(i-1)}+(i-2)+(i+1), \\
  & \multicolumn{1}{c}{\quad h-2 \text{ copies}}
  \end{array}
\end{equation*}
\begin{equation*}
  \begin{array}{rrrr}
  &h(i+1)+1=\underbrace{(i+1)+\cdots +(i+1)}+(i-1)+(i+4). \\
  & \multicolumn{1}{c}{\quad  h-2 \text{ copies}}
  \end{array}
\end{equation*}
For $2\leq m\leq 2h-1$, let $r_m$ be the least nonnegative residue of $m$ modulo $2$. Then $\displaystyle 1\leq \left\lfloor \frac{m-r_m}{2}\right\rfloor\leq h-1$, we have
\begin{equation*}
\hskip -3cm  \begin{array}{rrrr}
  &h(i-1)+m=\underbrace{(i-1)+\cdots +(i-1)}+(i-1-r_m)+&\underbrace{(i+1)+\cdots+(i+1)}+(i+1+2r_m), \\
  & \multicolumn{1}{c}{\quad\quad \quad\quad h-1-\frac{m-r_m}{2}\text{ copies}}&  \multicolumn{2}{c}{ \quad\quad\frac{m-r_m}{2}-1 \text{ copies}}
  \end{array}
\end{equation*}
\begin{equation*}
\hskip -3.5cm   \begin{array}{lrrrl}
  h(i+1)+m=&\underbrace{(i+1)+\cdots +(i+1)}+&\underbrace{(i+3)+\cdots+(i+3)}+(i+3+r_m), \\
  & \multicolumn{1}{c}{h-\frac{m-r_m}{2} \text{ copies}}&  \multicolumn{2}{c}{\frac{m-r_m}{2}-1 \text{ copies}}
  \end{array}
\end{equation*}

Hence, $|hA^{(N)}|=hk+h+1$.

This completes the proof of Proposition \ref{prop3}.
\end{proof}

\begin{prop}\label{prop4} Let $h\geq 2$, $k\geq 5$ be positive integers and $A^{(N)}=[0,k+1]\backslash\{i,j\}$ for some $i\in[1,k-2], j\geq i+3$.

\noindent(1) If $i=1$ and $j=k+1$, then $\left|hA^{(N)}\right|=hk$;

\noindent(2) If $i=1$ and $j=k$, then $\left|hA^{(N)}\right|=hk+h-1$;

\noindent(3) If $i=1$, $4\leq j\leq k-1$; or $2\leq i\leq k-3$, $j=k$, then $\left|hA^{(N)}\right|=hk+h$;

\noindent(4) If $i=k-2$; or $2\leq i\leq k-3$, $j=k+1$, then $\left|hA^{(N)}\right|=hk+1$.

\noindent(5) If $2\leq i\leq k-3$ and $j\leq k-1$, then $\left|hA^{(N)}\right|=hk+h+1$.

\end{prop}

\begin{proof} (1) If $i=1$ and $j=k+1$, then $A^{(N)}=\{0\}\cup[2,k]$. By Proposition \ref{prop1}(2), we have $\left|hA^{(N)}\right|=hk$.

(2) If $i=1$ and $j=k$, then $A^{(N)}=\{0\}\cup [2,k-1]\cup \{k+1\}$.
By the proof of Proposition \ref{prop1}(2), we have
$\{0\}\cup [2,hk-h]\subset hA^{(N)}.$

For $1\leq m\leq 2h-2$, let $r_m$ be the least nonnegative residue of $m$ modulo $2$. Then
\begin{equation*}
\hskip -3.6cm \begin{array}{lrrrl}
 h(k-1)+m=&\underbrace{(k-1)+\cdots+(k-1)}+&\underbrace{(k+1)+\cdots+(k+1)}+(k-1-r_m). \\
  & \multicolumn{1}{c}{ h-1-\frac{m+r_m}{2}\text{ copies}}&  \multicolumn{2}{c}{ \quad\frac{m+r_m}{2} \text{ copies}}
  \end{array}
\end{equation*}
Noting that $hk+h-1\not\in hA^{(N)}$, we have $|hA^{(N)}|=hk+h-1$.

(3) If $i=1$ and $4\leq j\leq k-1$, then $$A^{(N)}=\{0\}\cup[2,j-1]\cup[j+1,k+1].$$ By the proof of Proposition \ref{prop1}(2) we have
$$\{0\}\cup [2,hj-h]\subset hA^{(N)}.$$
Noting that $$[j-2,j-1]\cup [j+1,j+2]\subset A^{(N)},$$ by the proof of Lemma \ref{lem2} we have
$[hj-2h,hj+2h]\subset hA^{(N)}$.

 Hence, $|hA^{(N)}|=hk+h$.

If $2\leq i\leq k-3$ and $j=k$, then $$A^{(N)}=[0,i-1]\cup[i+1,k-1]\cup\{k+1\}:=A_1\cup \{k+1\}.$$
By Lemma \ref{lem2}, we have $hA_1=[0,h(k-1)]$. By the proof of Proposition \ref{prop4}(2), we have
$$[hk-h+1,hk+h-2]\subset hA^{(N)}\text{ and } hk+h-1\notin hA^{(N)}.$$
Hence, $|hA^{(N)}|=hk+h$.

(4) If $i=k-2$, then $A^{(N)}=[0,k-3]\cup\{k-1,k\}.$ By Lemma \ref{lem2} we have $|hA^{(N)}|=hk+1$.

If $2\leq i\leq k-3$ and $j=k+1$, then $$A^{(N)}=[0,i-1]\cup[i+1,k].$$
By Lemma \ref{lem2}, we have $hA^{(N)}=[0,hk]$, thus $|hA^{(N)}|=hk+1$.

(5) If $2\leq i\leq k-3$ and $j\leq k-1$, then $$A^{(N)}=[0,i-1]\cup[i+1,j-1]\cup[j+1,k+1].$$
By Lemma \ref{lem2} we have $[0,h(j-1)]\subset hA^{(N)}$.
Noting that $$[j-2,j-1]\cup [j+1,j+2]\subset A^{(N)},$$ by the proof of Lemma \ref{lem2} we have
$[hj-2h,hj+2h]\subset hA^{(N)}$.

Hence, $|hA^{(N)}|=hk+h+1$.

This completes the proof of Proposition \ref{prop4}.
\end{proof}

\section{Proof of Theorem \ref{thm1}}

If $hk-h+1<|hA|\leq hk+h-2$, then by Lemma \ref{2.1}, we have $A^{(N)}=[0,k]\backslash\{i\}$ for some $i\in[1,k]$.
 By Proposition \ref{prop1}, we have $|hA|=hk$ or $|hA|=hk+1$.

Again by Proposition \ref{prop1}, we have $|hA|=hk$ if and only if $A^{(N)}=[0,k]\setminus \{i\}$  with $i=1$ or $k-1$;
 $|hA|=hk+1$ if and only if $A^{(N)}=[0,k]\setminus \{i\}$ with some $2\leq i\leq k-2$.

This completes the proof of Theorem \ref{thm1}.

\section{Proof of Theorem \ref{thm2}}

If $hk+h-2<|hA|\leq hk+2h-3$, then by Lemma \ref{2.1}, we have $A^{(N)}=[0,k+1]\backslash\{i,j\}$ for some $1\leq i<j\leq k+1$.
 By Proposition \ref{prop2}-\ref{prop4}, we have $|hA|=hk+h-1$, $hk+h$ or $|hA|=hk+h+1$.

Again by Proposition \ref{prop1}, we have (a), (b) and (c).

This completes the proof of Theorem \ref{thm2}.
\vskip 1cm
\noindent{\bf Acknowledgment} We are grateful to Professor Yong-Gao Chen for his useful discussion and the anonymous referee for his/her valuable suggestions on improving the previous version of this work.

\end{document}